\documentclass[12pt,twoside]{amsart}
\usepackage{amsthm,amsfonts,amssymb,amscd,euscript}

\usepackage[matrix,arrow]{xy}

\author{Florin Ambro} 
\address{ Institute of Mathematics ``Simion Stoilow'' of the Romanian Academy\\
P.O. BOX 1-764, RO-014700 Bucharest\\ 
Romania.}
\email{florin.ambro@imar.ro}

\newcommand{\Q}{{\mathbb Q}}
\newcommand{\Z}{{\mathbb Z}}
\newcommand{\N}{{\mathbb N}}
\newcommand{\R}{{\mathbb R}}

\newcommand{\bP}{{\mathbb P}} % projective space

\newcommand{\cA}{{\mathcal A}}

\newcommand{\cF}{{\mathcal F}}
\newcommand{\cI}{{\mathcal I}}

\newcommand{\cO}{{\mathcal O}}

\newcommand{\Bs}{\operatorname{Bs}}

\newcommand{\ind}{\operatorname{index}}

\newcommand{\Supp}{\operatorname{supp}}

\theoremstyle{plain}
\newtheorem{thm}{Theorem}[section]

\newtheorem{lem}[thm]{Lemma}
\newtheorem{cor}[thm]{Corollary}

\newtheorem{prop}[thm]{Proposition}

\theoremstyle{definition}

\newtheorem{rem}[thm]{Remark}

\newtheorem{ack}{Acknowledgments}   

\theoremstyle{remark}

\begin{document}

\bibliographystyle{amsalpha+}
\title{Successive vanishing on curves}
%%\date{August 23, 2011}
\maketitle

\dedicatory{Dedicated to the memory of Lucian B\u{a}descu on the occasion of his 80th birthday}

\begin{abstract} 
	We investigate the emptiness of adjoint linear systems associated to successive multiples of a given positive divisor with real coefficients.
\end{abstract}

%\tableofcontents

%%%%%%%%%%%%%%%%%%%%%%%%%%%%%%%%%%%%%%%%%%%%%%%%%%%%%%%%%%%%%%%%%%%%%%
%%% Document name: SuccesiveVanishingCurves.tex
%%% Last modified: August 23, 2011
%%%%%%%%%%%%%%%%%%%%%%%%%%%%%%%%%%%%%%%%%%%%%%%%%%%%%%%%%%%%%%%%%%%%%%

\footnotetext[1]{2020 Mathematics Subject Classification. Primary: 14E30. Secondary: 14C20.}

%%%%%%%%%%%%%%%%%%%%%%%%%%%%%%%%%%%%%%%%%%%%%%%%%%%%%%%
\setcounter{section}{-1}
%%%%%%%%%%%%%%%%%%%%%%%%%%%%%%%%%%%%%%%
%%%%%%%%%%%%%%%%%%%%%%%%%%%%%%%%%%%%%%%

\section{Introduction}

%%%%%%%%%%%%%%%%%%%%%%%%%%%%%%%%%%%%%%%
%%%%%%%%%%%%%%%%%%%%%%%%%%%%%%%%%%%%%%%

We consider in this note the following question:
consider triples $(X,L,N)$, where $X$ is a complex projective nonsingular algebraic variety, $L$ is a nef and big $\R$-divisor on $X$, 
$\Supp(\lceil L\rceil-L)$ is a normal crossings divisor on $X$, and  $N$ is a positive integer such that
$$
|\lceil K_X+iL\rceil|=\emptyset \text{ for all } 1\le i\le N.
$$
Given $N$ (sufficiently large), can we classify the pairs $(X,L)$ with this property? What can be said about $L$ as $N$ approaches $+\infty$?
For example, a necessary condition on curves is $N\deg(L)\le 1$.
Note that the question does not change if we replace $X$ by a higher 
model and $L$ by its pullback. Also, we may suppose $X$ admits no fibration such that the restriction of $L$ to the general fiber satisfies the same properties.

Here $\lceil L\rceil$ denotes the round up of an $\R$-divisor, defined componentwise. If $\lceil L\rceil-L=0$, then $N\le \dim X$ by Kawamata-Viehweg vanishing and the classical argument that a polynomial has no more roots than its degree. So the difficulty seems to be hidden in the fractional part $\lceil L\rceil-L$. If the coefficients of $\lceil L\rceil-L$ are rational with bounded denominators, then $N$ is again bounded above. 

Our motivation is to understand when $|iK_X|=\emptyset$ for all $1\le i\le N$, where $X$ is a complex projective nonsingular variety of general type. It is known that $N$ is bounded above only in terms of $\dim X$~\cite{HM06,Tak06}. A positive answer to the question might give explicit bounds. The relation with the question are the inclusions
$$
\Gamma(X_l,\lceil K_{X_l}+(i-1)\frac{M_l}{l}\rceil)\subseteq \Gamma(X,iK_X)
$$
where $X_l\to X$ is a sufficiently high resolution and $M_l$ is the mobile part on $X_l$ of the linear system $|lK_X|$. The inclusions are equalities if $l$ is divisible by $i$. 

 In this note, we solve the question in dimension one (Theorem~\ref{v2}),
and give an application when $L$ is a positive multiple of a log divisor (Theorem~\ref{nvl}). We also solve the similar problem when successive adjoint linear systems contain a given point in the base locus (Theorem~\ref{4.9}). Other variants are possible, such as failure of successive adjoint linear systems to be mobile, to generate a fixed order of jets at a given point, etc. 

Two observations come out from the curve case.
First, for $N>1$, only the highest two coefficients of $\lceil L\rceil-L$ matter. Second, the Farey set of given order appears naturally. In fact, effective non-vanishing properties for $\R$-divisors can be restated in terms of divisors with coefficients in the Farey set of a given order.
A similar successive failure argument was used by Shokurov~\cite[Example 5.2.1]{Sho93} to construct lc $n$-complements on curves, with $n\in \{1,2,3,4,6\}$. In his setup, with $L=-K-B$ and $i$ starting at $2$, only the highest four coefficients of $B$ matter,
and suffices to consider coefficients of $B$ only in the Farey set of order $5$. 

It is very likely that our question can be solved for surfaces (cf.~\cite{CCZ05}). 

\begin{ack} 
This work was finalized while visiting Yau Mathematical Sciences Center, Tsinghua University. I would like to thank Caucher Birkar for the kind invitation.
\end{ack}

%%%%%%%%%%%%%%%%%%%%%%%%%%%%%%%%%%%%%%%
%%%%%%%%%%%%%%%%%%%%%%%%%%%%%%%%%%%%%%%

\section{Estimates}

%%%%%%%%%%%%%%%%%%%%%%%%%%%%%%%%%%%%%%%
%%%%%%%%%%%%%%%%%%%%%%%%%%%%%%%%%%%%%%%

For a positive integer $N$, consider the Farey set of order $N$ defined by
$$
\cF_N=[0,1]\cap \cup_{i=1}^N \frac{1}{i}\Z.
$$

The following properties hold:
\begin{itemize}
\item $x\in \cF_N$ if and only if $1-x\in \cF_N$.
\item The finite set $\cF_N$ decomposes the interval $[0,1)$ into finitely many 
disjoint intervals $[x,x')$. For $\delta\in [0,1)$, the unique interval which contains $\delta$ is determined by the formulas
$x=\max_{1\le i\le N}\frac{\lfloor i\delta\rfloor}{i},x'=\min_{1\le i\le N}\frac{1+\lfloor i\delta\rfloor}{i}$. Denote $x'$ by $\delta^+_N$.
\end{itemize}

\begin{lem} Let $0\le x<x'\le 1$. Then $(x,x')\cap \cF_N=\emptyset$
if and only if $\lfloor ix\rfloor+\lfloor i(1-x')\rfloor=i-1$ for every $1\le i\le N$.
\end{lem}

\begin{proof}
We have 
$
\lfloor ix\rfloor+\lfloor i(1-x')\rfloor\le ix+i(1-x')<i. 
$
Therefore $\lfloor ix\rfloor+\lfloor i(1-x')\rfloor\le i-1$.

We have $(ix,ix')\cap \Z=\emptyset$ if and only if $\lfloor ix\rfloor+1\ge \lceil ix'\rceil$, that is 
$
\lfloor ix\rfloor+\lfloor i(1-x')\rfloor\ge i-1.
$
Therefore $(ix,ix')\cap \Z=\emptyset$ if and only if $\lfloor ix\rfloor+\lfloor i(1-x')\rfloor= i-1$.

Finally, $(x,x')\cap \cF_N=\emptyset$ if and only if $(ix,ix')\cap \Z=\emptyset$ for every $1\le i\le N$.
\end{proof}

\begin{lem}\label{ari} Let $x<x'$ be two consecutive elements of $\cF_N$. Then:
\begin{itemize}
\item[1)] $x'-x\le \frac{1}{N}$. 
\item[2)] If $x'-x\ge \frac{1}{N+1}$, then $x=0$ or $x'=1$.
\end{itemize}
\end{lem}

\begin{proof} We have two cases. Either the $\cF_N$-interval is 
	$[0,\frac{1}{N})$ or $[\frac{N-1}{N},1)$, of length $\frac{1}{N}$.
	Or $x=\frac{p}{q},x'=\frac{p'}{q'}$, where $p,q,p',q'$ are positive integers such that $p'q-pq=1$, $\min(q,q')\ge 2$ and $\max(q,q')\le N<q+q'$. We have $(q-1)(q'-1)>1$, since it is not possible that both $q$ and $q'$ equal $2$. Therefore $qq'\ge q+q'+1\ge N+2$. Then
	$x'-x=\frac{1}{qq'}\le \frac{1}{N+2}$.
\end{proof}

\begin{lem}\label{nd}
Let $N\ge 2$, $0\le b<1$, $\frac{1}{2}\le \delta<\frac{N-1+b}{N}$.
Then there exist $1\le p<q\le N$ such that $\delta<\frac{p+b}{q}$ and $\frac{p}{q}-\delta<\frac{1}{N+1}$.
\end{lem}

\begin{proof}
Suppose $\delta\ge \frac{N-1}{N}$. Then we can take $p=N-1,q=N$.
Suppose $\delta<\frac{N-1}{N}$. Let $\delta\in [x,x')$ be the unique half open 
$\Z_N$-interval which contains it. We have $0<x<x'=\frac{p}{q}<1$. By 
Lemma~\ref{ari}.(2), $x'-x<\frac{1}{N+1}$. Then 
$\delta<\frac{p}{q}\le \frac{p+b}{q}$ and $\frac{p}{q}-\delta\le x'-x<\frac{1}{N+1}$.
\end{proof}

\begin{lem}\label{eac} Let $N\ge 1$ and $b\in [0,1)$.
\begin{itemize}
\item[1)] $\frac{N-1+b}{N}\le x<1$ if and only if $\lfloor ix-b\rfloor=i-1$ for all $1\le i\le N$.
\item[2)] If $N\ge 2$ and $\frac{N-2+b}{N-1}\le x<\frac{N-1+b}{N}$, 
then $\lfloor Nx-b\rfloor=N-2$.
\end{itemize}
\end{lem}

\begin{proof} 1) The implication $\Longleftarrow$ is clear. Consider the converse.
Let $1\le i\le N$. We have 
$
i-\frac{i+(N-i)b}{N}\le ix-b<i-b.
$ 
Since $b<1$, $i+(N-i)b<i$. 
Therefore $i-1<ix-b<i$. Therefore $\lfloor ix-b\rfloor=i-1$.

2) We have $N-1-\frac{1-b}{N-1}\le Nx-b<N-1$. Therefore $\lfloor Nx-b\rfloor=N-2$.
\end{proof}

\begin{prop}\label{crt}
Let $0\le b\le \delta<1$, $0\le b'\le \delta'<1$, $\delta'\le \delta$ and $\delta+\delta'\le 1$. 
Let $N\ge 2$. Then 
$$
\lfloor i\delta-b\rfloor+\lfloor i\delta'-b'\rfloor \ge i-1 \text{ for all }1\le i\le N
$$
if and only if one of the following holds:
\begin{itemize}
\item[a)] $\delta\ge \frac{N-1+b}{N}$, or
\item[b)] $\frac{1+b}{2}\le \delta<\frac{N-1+b}{N}$ and 
$\delta'\ge \max\{\frac{q-p+b'}{q};1\le p<q\le N,\delta<\frac{p+b}{q}\}$.
\end{itemize}
Moreover, $\delta\ge 1-\frac{1}{N}$ in case a). In case b), $\delta+\delta'>1-\frac{1}{N+1}$ and 
$\delta'\ge \frac{1}{N}$ (in particular, $\delta+2\delta'>1$).
And $\delta'=\delta$ if and only if $\delta'=\delta=\frac{1}{2},b=0$.
\end{prop}

\begin{proof}
Let $N=2$. The system of inequalities becomes $\lfloor 2\delta-b\rfloor+\lfloor 2\delta'-b'\rfloor \ge 1$.
That is $2\delta-b\ge 1$ or $2\delta'-b'\ge 1$. If $2\delta-b\ge 1$, we are in case a). If 
$2\delta'-b'\ge 1$, then $\delta'\ge \frac{1}{2}$. Then $\delta\ge \frac{1}{2}$. Then 
$\delta=\delta'=\frac{1}{2}$ and $b=b'=0$. We are in case a).

Let $N\ge 3$. Suppose $\delta\ge \frac{N-1+b}{N}$. By Lemma~\ref{eac}.1), 
$\lfloor i\delta-b\rfloor= i-1$ for all $1\le i\le N$. The system of inequalities is satisfied.
Suppose $\delta<\frac{N-1+b}{N}$. From the case $N=2$, we obtain 
$$
\frac{1+b}{2}\le \delta<\frac{N-1+b}{N}.
$$
Suppose the system of inequalities if satisfied. Let $1\le p<q\le N$ with $\delta<\frac{p+b}{q}$.
Then $p>q\delta-b$, that is $p-1\ge \lfloor q\delta-b\rfloor$. The inequality for $i=q$ gives
$p-1+\lfloor q\delta'-b'\rfloor \ge q-1$. Therefore $\lfloor q\delta'-b'\rfloor\ge q-p$. That is 
$q\delta'-b'\ge q-p$. Therefore $\delta'\ge \frac{q-p+b'}{q}$. So b) holds.

Conversely, suppose b) holds. Let $1\le i\le N$. Let $p=1+\lfloor i\delta-b\rfloor$.
If $p=i-1$, then $\lfloor i\delta-b\rfloor+\lfloor i\delta'-b'\rfloor \ge \lfloor i\delta-b\rfloor= i-1$.
If $p<i$, then $1\le p<i\le N$ and $\delta<\frac{p+b}{i}$. By b) for $q=i$, we deduce
$\delta'\ge \frac{i-p+b'}{i}$. Then $\lfloor i\delta'-b'\rfloor\ge i-p$. Therefore 
$\lfloor i\delta-b\rfloor+\lfloor i\delta'-b'\rfloor \ge p-1+i-p=i-1$. We conclude that the system
of inequalities holds is equivalent to a) or b).

In case a), $\delta\ge \frac{N-1+b}{N}\ge 1-\frac{1}{N}$. Consider case b). 
We have $N\ge 3$. By Lemma~\ref{nd}, there exists $1\le p<q\le N$ such that 
$\delta<\frac{p+b}{q}$ and $\frac{p}{q}-\delta<\frac{1}{N+1}$. By b),
$
\delta'\ge \frac{q-p+b'}{q}.
$
Therefore 
$$
\delta+\delta'\ge \delta+\frac{q-p}{q}>\frac{p}{q}-\frac{1}{N+1}+\frac{q-p}{q}= 1-\frac{1}{N+1}.
$$
From $\delta<\frac{N-1+b}{N}$ we deduce $\delta'\ge \frac{1+b'}{N}$. In particular
$\delta'\ge \frac{1}{N}$. In particular, $\delta+2\delta'>1+\frac{1}{N(N+1)}$.

Suppose $\delta'=\delta$. Since $\delta\ge \frac{1+b}{2}$ and $\delta+\delta'\le 1$, we
deduce $\delta=\delta'=\frac{1}{2}$ and $b=0$. 
\end{proof}

\begin{rem} Let $b\in [0,1)$, let $N\ge 1$.
Consider the totally ordered finite set $[0,1)\cap \{\frac{p+b}{q};1\le q\le N,p\in \N\}$.
The maximal element is $\frac{N-1+b}{N}$. If $N=1$, this is the only element. If 
$N\ge 2$, the next largest element is $\frac{N-2+b}{N-1}$.
\end{rem}

\begin{cor}\label{crtNoB}
Let $0\le \delta'\le \delta<1,\delta+\delta'\le 1$, let $N\ge 2$. Then 
$
\lfloor i\delta\rfloor+\lfloor i\delta'\rfloor \ge i-1 \text{ for all }1\le i\le N
$
if and only if $\delta^+_N+\delta'\ge 1$.
\end{cor}

\begin{prop}\label{crtdiv} Let $0\le B\le \Delta$ be $\R$-divisors on a nonsingular curve such that
$\lfloor \Delta\rfloor=0$ and $\deg\Delta\le 1$. 
Let $P$ be a point where $\Delta$ attains its maximal multiplicity,
let $P'$ be a point where $\Delta'=\Delta-\delta_P P$ attains its maximal
multiplicity. Denote $\Delta''=\Delta-\delta_P P-\delta_{P'} P'$. Let $N\ge 2$. 

Then 
$
\deg \lfloor i\Delta-B\rfloor\ge i-1 \text{ for all }1\le i\le N
$ 
if and only if $\lfloor N\Delta''\rfloor=0$ and 
$\lfloor i\delta_P-b_P\rfloor+\lfloor i\delta_{P'}-b_{P'}\rfloor\ge i-1$ for all $1\le i\le N$.
\end{prop}

\begin{proof} Suffices to show $\lfloor N\Delta''\rfloor=0$. We use induction on $N$.

Let $N=2$. Suppose by contradiction $\lfloor 2\Delta''\rfloor\ne 0$. Then 
$\Delta''$ has a coefficient $\delta''\ge \frac{1}{2}$. Then $\delta_P\ge \delta_{P'}\ge \delta''\ge 
\frac{1}{2}$. Then $\deg\Delta\ge \frac{3}{2}>1$, a contradiction. Therefore 
$\lfloor 2\Delta''\rfloor=0$.

Let $N>2$. Suppose $\delta_P\ge \frac{N-1+b_P}{N}$. 
In particular, $\deg \Delta'\le \frac{1}{N}$. If $\lfloor N\Delta'\rfloor\ne 0$, then 
$\delta_{P'}\ge \frac{1}{N}$. Therefore $\Delta=\frac{N-1}{N}P+\frac{1}{N}P'$ and $B\le \frac{1}{N}P'$.
Here $\Delta''=0$. If $\lfloor N\Delta'\rfloor=0$, then $\lfloor N\Delta''\rfloor=0$.

Suppose $\delta_P<\frac{N-1+b_P}{N}$.

Case $\frac{N-2+b_P}{N-1}\le \delta_P$. Then $\lfloor N\delta_P-b_P\rfloor=N-2$.
Denote $B'=B-b_PP$. Our system of inequalities becomes 
$\deg \lfloor N\Delta'-B'\rfloor\ge 1$. That is $\lfloor N\Delta'-B'\rfloor\ne 0$. That is 
$N\delta_Q-b_Q\ge 1$ at some point $Q\in \Supp \Delta'$. We have 
$$
\delta_P+\delta_{P'}\ge \delta_P+\delta_Q\ge \frac{N-2}{N-1}+\frac{1}{N}=1-\frac{1}{(N-1)N}.
$$
From $(N-1)N>N$, we deduce $\delta_P+\delta_{P'}>1-\frac{1}{N}$.
Therefore $\deg\Delta''<\frac{1}{N}$. Therefore $\lfloor N\Delta''\rfloor=0$.

Case $\delta_P<\frac{N-2+b_P}{N-1}$. By induction, 
$\lfloor i\delta_P-b_P\rfloor+\lfloor i\delta_{P'}-b_{P'}\rfloor\ge i-1$ for all $1\le i\le N-1$.
We are in case $b)_{N-1}$ of Proposition~\ref{crt}. Therefore $\delta_P+\delta_{P'}>1-\frac{1}{(N-1)+1}$. 
Therefore $\deg\Delta''<\frac{1}{N}$. Therefore $\lfloor N\Delta''\rfloor=0$.
\end{proof}

\begin{cor} 
Suppose $\Delta\ge 0$, $\lfloor \Delta\rfloor=0$ and $\deg\Delta\le 1$. 
Let $P$ be a point where $\Delta$ attains its maximal multiplicity $\delta$.
Let $P'$ be a point where $\Delta'=\Delta-\delta P$ attains its maximal
multiplicity $\delta'$. Denote $\Delta''=\Delta-\delta P-\delta' P'$. Let $N\ge 2$. 

Then 
$
\deg \lfloor i\Delta\rfloor\ge i-1 \text{ for all }1\le i\le N
$ 
if and only if $\lfloor N\Delta''\rfloor=0$ and $\delta^+_N+\delta'\ge 1$. 
\end{cor}

%%%%%%%%%%%%%%%%%%%%%%%%%%%%%%%%%%%%%%%

\subsection{Case $N=+\infty$}

%%%%%%%%%%%%%%%%%%%%%%%%%%%%%%%%%%%%%%%

Let $\Delta$ be an effective $\R$-divisor on a nonsingular curve. Let
$\delta$ be the higher multiplicity of $\Delta$, attained at $P$ say.
Let $\delta'$ be the highest multiplicity of $\Delta'=\Delta-\delta P$. Note $\delta\ge \delta'$.

\begin{lem}\label{1.11} Suppose $\deg\Delta\le 1$. Then 
$\deg\lfloor i\Delta\rfloor\ge i-1$ for all $i\ge 1$ if and only if 
$\Delta=P$ or $\Delta=\delta P+(1-\delta)P'$ for some $\delta\in [\frac{1}{2},1)$.
\end{lem}

\begin{proof} Let $\lfloor \Delta\rfloor\ne 0$. That is $\Delta\ge P$ for some $P$.
That is $\Delta=P$. The inequalities are satisfied. Let $\lfloor \Delta\rfloor=0$. 
From above, the highest two coefficients of $\Delta$ satisfy
$\delta+\delta'\ge 1-\frac{1}{N}$, with $N\to\infty$. Therefore $\delta+\delta'\ge 1$.
From $\deg\Delta\le 1$, we deduce $\delta'=1-\delta$ and $\Delta=\delta P+(1-\delta)P'$.
Here $\lfloor i\delta\rfloor+\lfloor i-i\delta\rfloor=i-(\lceil i\delta\rceil-\lfloor i\delta\rfloor )\ge i-1$.
\end{proof}

\begin{cor}  $\deg\lfloor i\Delta\rfloor=i-1$ for all $i\ge 1$ 
if and only if $\Delta=\delta P+(1-\delta)P'$ for some $\delta\in [\frac{1}{2},1]\setminus \Q$.
\end{cor}

\begin{lem}\label{1.12} 
$\deg\Delta\le 1$, $\deg\lfloor i\Delta\rfloor\ge i-1 \ (1\le i\le N)$ if and only if $\delta^+_N+\delta'\ge 1$.
\end{lem}

\begin{cor}
$\deg\Delta\le 1$, $\deg\lfloor i\Delta\rfloor=i-1 \ (1\le i\le N)$ if and only if
\begin{itemize}
\item $\delta\in [\frac{N-1}{N},1)$, $\lfloor N\Delta'\rfloor=0$, or 
\item $\delta'\ge 1-\delta^+_N$ and either $\deg \Delta<1$, or 
$\deg\Delta=1$ and $i\Delta\notin \Z$ for all $1\le i\le N$.
\end{itemize}
\end{cor}

\begin{proof}
$ \deg\lfloor i\Delta\rfloor\le \deg i\Delta \le i$. Equality if and only if $\deg\Delta=1,i\Delta\in \Z$.
\end{proof}

%%%%%%%%%%%%%%%%%%%%%%%%%%%%%%%%%%%%%%%

\subsection{Fractions with bounded numerators}

%%%%%%%%%%%%%%%%%%%%%%%%%%%%%%%%%%%%%%%

For $l\in \Z_{\ge 1}$, define $\cA_l$ to be the set of rational numbers 
$x\in (0,1)$ which admit a representation $x=\frac{p}{q}$ with $p,q$ positive integers and $p\le l$. 

Note that $x\in (0,1)$ belongs to $\cA_l$ if and only if $\{\frac{1}{x}\}$ belongs to the Farey set of order $l$. If $x\in \cA_l$, then $\{l'x\}\in \cA_{l'l}$. We have inclusions $\cA_l\subseteq \cA_{l+1}$, and  
$
\max \cA_l=\frac{l}{l+1}.
$

The set $1-\cA_l$ is related to the hyperstandard set~\cite{PS09} associated to $\cF_{l+1}$.

\begin{lem}\label{fs1}
Let $x\ge y$ belong to $\cA_l$, with $x+y=1$. Then 
$x=\frac{p}{q}$ and $y=\frac{q-p}{q}$ for some 
$1\le p\le l, p<q\le 2p, \gcd(p,q)=1$.
In particular, $q\le 2l$.
\end{lem}

\begin{lem}
	Let $x\in \cA_l$. Let $N$ be the unique positive integer such that $x^+_{N+1}<1\le x^+_N$. Then $N\le l+1$, and equality is attained only if $x=\frac{l}{l+1}$.
\end{lem}

\begin{proof} Note that $N=\lfloor \frac{1}{1-x} \rfloor$.
	Let $x=\frac{p}{q}$ be the reduced form, where $p\le l$.
	Since $x<1$, $j=q-p$ is a positive integer. Then $N=\lfloor \frac{1}{1-x} \rfloor=1+\lfloor \frac{p}{j}\rfloor \le 1+p$. Equality holds if and only if $j=1$.
\end{proof}

\begin{lem}
	Let $1>x\ge y>0$ with $x+y<1$. Let $N$ be the unique positive integer such that $x^+_{N+1}+y<1\le x^+_N+y$. Suppose $x,y$ are rational, with reduced forms $x=\frac{p}{q},y=\frac{p'}{q'}$. Then $N\le (p+1)(p'+1)$, and equality is attained if and only if
	$$
	(x,y)=(\frac{p}{p+1},\frac{p'}{1+p'(p+1)}).
	$$
\end{lem}

\begin{proof}
	We have $x<1$. Then $x=\frac{p}{p+j}$ for some positive integer $j$.
	The inequality $y<1-x$ is equivalent to $q'\ge 1+p'+z$, where 
	$z=\lfloor \frac{pp'}{j}\rfloor$. Denote $y'=\frac{p'}{1+p'+z}$. Then $y\le y'<1-x$. 
	
	Note that $N+1$ is the smallest index of a rational number contained in the interval $(y,1-x)$. We have two cases.
	If $y<y'$, then $N+1\le 1+p' + z$. If $y=y'$, then $N+1\le 1+p' + z+p+j$.
	We conclude that $N\le p+p'+j+z$.
	
	We claim that $j+z \ge 1+pp'$ implies
	$j=1$ or $pp'\le j$. Indeed, we deduce $j+\frac{pp'}{j}\ge 1+pp'$.
	If $j>1$, we have $pp'\le j$.
	
	We claim that $x<\frac{1}{2}$ implies $N=1$. Indeed, suppose 
	$N\ge 2$. Then $x^+_N\le \frac{1}{2}$. Then $y\ge \frac{1}{2}$.
	Then $y>x$, a contradiction!
	
	For $x=\frac{1}{2}$, we have $p=j=1$. We have $x> \frac{1}{2}$ if and only if $p>j$. 
	
	We conclude that $N\le p+p'+1+pp'$, and equality is attained only if $j=1$ and $y=y'$.
\end{proof}

\begin{prop}\label{fl1}
Let $x\ge y$ in $\cA_l$ with $x+y<1$. Let $N$ be the unique positive integer such that $x^+_{N+1}+y<1\le x^+_N+y$. Then $N\le (l+1)^2$, and equality is attained only for 
$$
(x,y)=(\frac{l}{l+1},\frac{l}{l^2+l+1}).
$$
\end{prop}

%%%%%%%%%%%%%%%%%%%%%%%%%%%%%%%%%%%%%%%
%%%%%%%%%%%%%%%%%%%%%%%%%%%%%%%%%%%%%%%

\section{Succesive vanishing}

%%%%%%%%%%%%%%%%%%%%%%%%%%%%%%%%%%%%%%%
%%%%%%%%%%%%%%%%%%%%%%%%%%%%%%%%%%%%%%%

Let $C$ be a complex smooth projective curve, and $B,L$ $\R$-divisors on $C$ with $B\ge 0$ and $\deg L\ge 0$.

\begin{lem}\label{v1}
$|\lceil K+B+L\rceil|=\emptyset$ if and only if one of the following holds:
\begin{itemize}
\item[1)] $C=\bP^1, L\sim 0$, $B\le P$ for some point $P\in C$.
\item[2)] $C=\bP^1, L\sim Q-\Delta$, $\lfloor \Delta\rfloor=0$, $B\le \Delta$ for some point $Q\in C$. 
\item[3)] $C$ is an elliptic curve, $L\sim Q-P$, $B=0$, for 
some point $Q\ne P$.
\end{itemize}
In particular, $\deg L\le 1$.
\end{lem}

\begin{proof} By Riemann-Roch and Serre duality, we have 
$$
-h^0(-\lceil B+L\rceil)=g-1+\deg\lceil B+L\rceil.
$$
In particular, $g\le 1$. Note $\deg\lceil B+L\rceil\ge 0$, with
equality if and only if $B=0$ and $L$ is Cartier of degree $0$.
If $g=1$, then $B=0$, $L$ is Cartier of degree zero, and $h^0(-L)=0$
(case 3)).

Suppose $g=0$. Then $\deg\lceil B+L\rceil+h^0(-\lceil B+L\rceil)=1$.

If $\deg\lceil B+L\rceil=0$, then $B=0,L\sim 0$ (case 1)).
Else $\deg\lceil B+L\rceil=1$. Denote $\Delta=\lceil B+L\rceil-L$. 
In particular, $B\le \Delta$.

Case $\lfloor \Delta\rfloor\ne 0$. That is $\Delta\ge P$ for some 
$P\in C$. Since $\deg L\ge 0$ and $\deg \lceil B+L\rceil=1$, 
we obtain $0\ne B\le \Delta=P,L\sim 0$ (case 1)).

Case $\lfloor \Delta\rfloor=0$. Choose any point $Q\in C$. 
Then $L\sim Q-\Delta$ (case 2)).
\end{proof}

\begin{thm}\label{v2} 
Let $N\ge 2$. Then $|\lceil K+B+iL\rceil|=\emptyset$ for all $1\le i\le N$ 
if and only if one of the following holds:
\begin{itemize}
\item[1)] $C=\bP^1, L\sim 0$, $B\le P$ for some point $P\in C$.
\item[2)] $C=\bP^1, L\sim Q-\Delta$, $\lfloor \Delta\rfloor=0$, $B\le \Delta$
and $\deg\lfloor i\Delta-B\rfloor\ge i-1$ for all $1\le i\le N$.
\item[3)] $C$ is an elliptic curve, $L\sim Q-P$, $P\notin \Bs|iQ-(i-1)P|$ for all 
$1\le i\le N$, and $B=0$.
\end{itemize}
\end{thm}

\begin{proof} From $N=1$, we have three cases.

In case 1), $K+B+iL\sim -P$ for all $i\ge 1$. Therefore $|\lceil K+B+iL\rceil|=\emptyset$ 
for all $i\ge 1$.

In case 2), $C$ is a rational curve, $L\sim Q-\Delta$ for some point $Q$, 
$\lfloor \Delta\rfloor=0$ and $B\le \Delta$. Note that $\Delta=\lceil L\rceil-L$, so 
$\Delta$ is an intrinsic invariant of the linear equivalence class of $L$. 
We have
$$
K+B+iL\sim K+iQ-(i\Delta-B)\sim (i-2)Q-(i\Delta-B).
$$
So $|\lceil K+B+iL\rceil|=\emptyset$ if and only if $\deg\lfloor i\Delta-B\rfloor\ge i-1$.

Note $\lceil K+\Delta+2L\rceil\sim 0$.

In case 3), $C$ is an elliptic curve, $L\sim Q-P$ and $|iL|=\emptyset$ for every 
$1\le i\le N$. The linear system $|iQ-(i-1)P|$ is fixed, so the last property is equivalent
to $P\notin \Bs|iQ-(i-1)P|$ for all $1\le i\le N$.
\end{proof}

The divisor $L$ may have non-integer coefficients only in case 2). And in this 
case, $\Delta=\lceil L\rceil-L$.

\begin{cor} Let $C=\bP^1$ and $L$ an $\R$-divisor with $\deg L\ge 0$. 
Then $\lceil K+iL\rceil=\emptyset$ for all $1\le i\le N$ 
if and only if there exists $a\in k(C)^\times$ such that
$$
(a)+L\le x'P'-xP,
$$ 
where $P\ne P'$ and $x<x'$ are consecutive elements of $\cF_N$. 
For each $N$, the maximal elements are rational, and finitely
many. For $N=1$, the maximal element is unique, equal to $P'$. 
For $N\ge 2$, there exists an integer $1\le i\le (N-1)N$ such that the linear system
$|iL^{max}|$ is free of degree $1$.
\end{cor}

\begin{thm} 
$|\lceil K+B+iL\rceil|=\emptyset$ for all $i\ge 1$ if and only if
one of the following holds:
\begin{itemize}
\item $C=\bP^1$, $L\sim 0$, $B\le P$ for some $P$.
\item $C=\bP^1$, $L\sim \epsilon (P_1-P_2)$, $P_1\ne P_2,\epsilon\in (0,1)$, 
and either $B=0$, or $\epsilon$ is rational of index $l$ and $0\ne B\le \frac{1}{l}P_1$.
\item $C$ is an elliptic curve, $L\sim P_1-P_2$, $|iP_1-(i-1)P_2|\ne P_2$ for all $i\ge 1$, $B=0$.
\end{itemize} 
\end{thm}

\begin{proof}
May suppose $C=\bP^1$.
Case $\lfloor \Delta\rfloor\ne 0$. That is $\delta\ge 1$.
Then $\delta=1$. Then $\Delta=P,L\sim 0$ (first case).

Case $\lfloor \Delta\rfloor=0$. Let $P$ be
the point of maximal multiplicity for $\Delta$
(assumed non-zero). Then $\delta<1$. Let 
$N(1-\delta)>1$. Then case a)$_N$ does not occur.
So we are in case b)$_N$. Therefore 
$$
\delta+\delta'>1-\frac{1}{N}.
$$
Letting $N\to \infty$ we obtain
$
\delta+\delta'\ge 1.
$
Then equality holds. Then 
$$
\Delta=\delta P+(1-\delta)P', L\sim (1-\delta)(P-P').
$$
We have $K+B+iL\sim K+iQ-(i\Delta-B)$. Therefore vanishing holds
up to $N$ if and only if 
$$
\lfloor i\delta-b\rfloor+\lfloor i(1-\delta)-b'\rfloor\ge i-1 \ (i\ge 1).
$$
This is equivalent to 
$$
b+b'\le 1-\{i\delta-b\}\ (i\ge 1).
$$
From $i=1$ we deduce $b'=0$. The system of inequalities becomes
$1-b\ge \{i\delta-b\}$ for all $i\ge 1$. If $b=0$, it is satisfied. If $b>0$,
it is satisfied if and only if $\delta$ is rational and $b\le \frac{1}{q}$
where $q$ is the index of $\delta$.
\end{proof}

%%%%%%%%%%%%%%%%%%%%%%%%%%%%%%%%%%%%%%%

\subsection{An application}

%%%%%%%%%%%%%%%%%%%%%%%%%%%%%%%%%%%%%%%
Let $C$ be a proper smooth curve, let $B$ be an effective $\R$-divisor on $C$ such that $\deg(K+B)\ge 0$.

- Suppose $\deg\lfloor B\rfloor\ge 1$. By Lemma~\ref{v1}, exactly one of the following holds:
\begin{itemize}
	\item[a)] $\deg\lfloor B\rfloor= 1$ and $r(K+B)\sim 0$ for some positive integer $r$, or
	\item[b)] $|\lceil K+\lfloor B\rfloor+n(K+B) \rceil|\ne \emptyset$ for all $n\ge 1$.
\end{itemize}
In case a), we may choose $r$ minimal with this property, and then $|\lceil K+\lfloor B\rfloor+n(K+B) \rceil|\ne \emptyset$ if and only if $r$ does not divide $n$.

- For the rest of this section, suppose $\lfloor B\rfloor=0$. That is, $B$ has coefficients in $[0,1)$.
Let $m\in \Z_{\ge 1}$. Suppose $N\ge 2$ is an integer such that 
$|\lceil K+im(K+B) \rceil|=\emptyset$ for every $1\le i\le N$.

By Theorem~\ref{v2}, $C=\bP^1$ and there are two possibilities:
\begin{itemize}
	\item[1)] $m(K+B)\sim 0$, or
	\item[2)] $m(K+B)\sim Q-\Delta$ where $\lfloor \Delta \rfloor=0$ and $\deg \lfloor i\Delta\rfloor\ge i-1$ for $1\le i\le N$. 
\end{itemize}
Consider case 2). Then $\Delta=\lceil mB\rceil-mB$.  Write $\Delta=\delta P+\delta'P'+\Delta''$, where $P,P'$ are distinct points not contained in the support of $\Delta''$, $1>\delta\ge \delta'\ge 0$ and $\delta'$ is greater or equal to the coefficients of $\Delta''$.
Since $K+B$ is nef, we have $\deg \Delta\le 1$. In particular, $\delta+\delta'\le 1$. We have two cases:
\begin{itemize}
	\item[2a)] Suppose $\delta+\delta'=1$. Then $\Delta=\delta P+(1-\delta)P'$ and $\frac{1}{2}\le \delta<1$. Since $Q\sim P$, we obtain 
	$m(K+B)\sim (1-\delta)(P-P')$.
	\item[2b)] Suppose $\delta+\delta'<1$. By Lem 1.13,  $\deg \lfloor i\Delta\rfloor\ge i-1$ for all $1\le i\le N$ if and only if 
	$1\le \delta^+_N+\delta'$ and $\lfloor N\Delta''\rfloor=0$.
\end{itemize}

By Lemma~\ref{fl1} and Proposition~\ref{fl1}, we obtain

\begin{thm}\label{nvl} Suppose the smallest two (possibly equal) non-zero coefficients of $\{mB\}$ are of the form $1-\frac{p}{q}$, for some positive integers $p,q$ with $p\le l$. Then exactly one of the following holds:
	\begin{itemize}
		\item[a)] $nm(K+B)\sim 0$ for some $1\le n\le 2l$, or
		\item[b)] $|\lceil K+nm(K+B) \rceil|\ne \emptyset$ for some $1\le n\le (l+1)^2+1$.
	\end{itemize}
\end{thm}

The inequality in b) is attained for $C=\bP^1$, $m=1$, $B=(1-\frac{l}{1+l})P+(1-\frac{l}{1+l+l^2})P'$.
This resembles examples considered in~\cite{Kol94}.

\begin{cor} Suppose the smallest two (possibly equal) non-zero coefficients of $B$ are of the form $1-\frac{1}{q}$, 
	for some positive integers $q$ (i.e. standard coefficients). Then exactly one of the following holds:
	\begin{itemize}
		\item[a)] $nm(K+B)\sim 0$ for some $1\le n\le 2$, or
		\item[b)] $|\lceil K+nm(K+B) \rceil|\ne \emptyset$ for some $1\le n\le 5$.
	\end{itemize}
\end{cor}

%%%%%%%%%%%%%%%%%%%%%%%%%%%%%%%%%%%%%%%
%%%%%%%%%%%%%%%%%%%%%%%%%%%%%%%%%%%%%%%

\section{Successive base point}

%%%%%%%%%%%%%%%%%%%%%%%%%%%%%%%%%%%%%%%
%%%%%%%%%%%%%%%%%%%%%%%%%%%%%%%%%%%%%%%

Let $C/k$ be a nonsingular projective algebraic curve, let $B,L$ be $\R$-divisors
such that $B\ge 0$ and $\deg L\ge 0$. Let $Q\in C$ be a closed point.

\begin{prop}
Let $D$ be a divisor on $C$. Then $Q\in \Bs|K+D|$ if and only if $Q\notin \Bs|Q-D|$.
\end{prop}

\begin{proof} From the short exact sequence $0\to \cI_Q(K+D)\to \cO_C(K+D)\to \cO_Q\to 0$,
we deduce that $Q\in \Bs|K+D|$ if and only if the homomorphism 
$
H^1(K+D-Q)\to H^1(K+D)
$ 
is not injective. By Serre duality, this means that $\Gamma(-D)\to \Gamma(-D+Q)$ is 
not surjective. That is $Q\notin \Bs|-D+Q|$.
\end{proof}

\begin{thm}\label{sp} $Q\in \Bs|\lceil K+B+L\rceil|$ if and only if one of the following holds:
\begin{itemize}
\item[1)] $L\sim Q-P\ (Q\ne P),B=0$.
\item[2)] $L\sim Q-P$, $0\ne B\le P$.
\item[3)] $L\sim Q-\Delta,\lfloor \Delta\rfloor=0,B\le \Delta$.
\end{itemize}
\end{thm}

\begin{proof} Our assumption is equivalent to $Q\notin \Bs|Q-\lceil B+L\rceil|$.
In particular, $\deg \lceil B+L\rceil \le 1$.

Case $\deg \lceil B+L\rceil=0$. Then $B=0$, $L$ has integer coefficients and has degree
zero. Then $D=P$ for some point $P\ne Q$. We are in case 1).

Case $\deg \lceil B+L\rceil=1$. Then $D=0$. So $\lceil B+L\rceil\sim Q$.
Denote $\Delta=\lceil B+L\rceil-L$. Then $B\le \Delta$, $L\sim Q-\Delta$.
The property $\lceil B+L\rceil\sim Q$ translates into $\lfloor \Delta-B\rfloor=0$.
If $\lfloor \Delta\rfloor=0$, this property holds (case 3)). If 
$\lfloor \Delta\rfloor\ne 0$, we deduce from $\deg \Delta\le 1$ that 
$\Delta=P$ for some point $P\in C$ and $0\ne B\le P$ (case 2)).
\end{proof}

Remarks:

- If $L_1\sim L_2$, then $\lceil L_1\rceil-L_1=\lceil L_2\rceil-L_2$ and $L_1-\lfloor L_1\rfloor=L_2-\lfloor L_2\rfloor$.

- $\deg\lceil L\rceil$ is $0$ in cases 1),2), and $1$ in case 3). In case 3), $\Delta=\lceil L\rceil-L$.

- We have $\lceil K+B+L\rceil\sim K+Q-P\ (Q\ne P)$ in case 1), and $\lceil K+B+L\rceil\sim K+Q$ in cases 2),3).

- $|K+Q|$ is empty if $C=\bP^1$, $|K|+Q$ if $g\ge 1$. 

- $|K+Q-P|\ (Q\ne P)$ is empty if $g\le 1$, $|K-P-P'|+Q+P'$ if $C$ is hyperelliptic of genus $g\ge 2$ (where $P+P'$
is the fiber of $C\to \bP^1$) , $|K-P|+Q$ if $C$ is non-hyperelliptic of genus $g\ge 2$. 

In particular, if the linear system $|\lceil K+B+L\rceil |$ is not empty, its fixed part has degree $0$, $1$ or $2$. 

- It follows that $B$ is a boundary. And $\lfloor B\rfloor\ne 0$ if and only if $B=P,L\sim Q-P$ (so $P$ is uniquely 
determined by $B$).

\begin{lem} $Q\in \Bs|\lceil K+B+iL\rceil|$ for $1\le i\le 2$ if and only if one of the following holds:
\begin{itemize}
\item[1)] $L\sim Q-P\ (Q\ne P),Q\notin \Bs|2P-Q|, B=0$.
\item[2)] $L\sim 0, 0\ne B\le P\sim Q$.
\item[3)] $L\sim Q-\Delta,\lfloor \Delta\rfloor=0, B\le \Delta, Q\notin \Bs|\lfloor 2\Delta-B\rfloor-Q|$.
\end{itemize}
\end{lem}

\begin{proof} From $N=1$, we have three cases:

1) $L\sim Q-P$, $Q\ne P$, $B=0$. The new condition is $Q\notin \Bs|Q-2(Q-P)|$. That is $Q\notin \Bs|2P-Q|$.

2) $L\sim Q-P$, $0\ne B\le P$. The new condition is $Q\notin \Bs|Q-\lceil B+2(Q-P)\rceil|$. That is 
$Q\notin \Bs|\lfloor 2P-B\rfloor -Q|$. In particular, $\deg \lfloor 2P-B\rfloor\ge 1$. That is $0\ne B\le \frac{1}{2}P$.
Then $\lfloor 2P-B\rfloor=P$. The condition becomes $Q\notin \Bs|P-Q|$. Therefore $P\sim Q$. Therefore $L\sim 0$.

3) $L\sim Q-\Delta$, $\lfloor \Delta\rfloor=0$, $B\le \Delta$. The new condition is $Q\notin \Bs|Q-\lceil B+2Q-2\Delta\rceil |$.
That is $Q\notin \Bs|\lfloor 2\Delta-B\rfloor-Q|$.
\end{proof}

\begin{thm} Let $N\ge 2$. Then $Q\in \Bs|\lceil K+B+iL\rceil|$ for $1\le i\le N$ if and only if one of the following holds:
\begin{itemize}
\item[1)] $L\sim Q-P\ (Q\ne P),Q\notin \cup_{i=2}^N\Bs|iP-(i-1)Q|, B=0$.
\item[2)] $L\sim 0, 0\ne B\le P\sim Q$.
\item[3)] $L\sim Q-\Delta,\lfloor \Delta\rfloor=0, B\le \Delta, Q\notin\cup_{i=2}^N\Bs|\lfloor i\Delta-B\rfloor-(i-1)Q|$.
\end{itemize}
\end{thm}

\underline{It remains to classify case 3): suppose $L\sim Q-\Delta,\lfloor \Delta\rfloor=0, B\le \Delta$}.

Note $\lfloor i\Delta-B\rfloor\le i\Delta-B\le i\Delta$ and $\deg\Delta\le 1$. Therefore 
$\deg \lfloor i\Delta-B\rfloor\le i$, and equality holds if and only if $B=0,\deg\Delta=1,iB\in \Z$. Therefore
$Q\notin \Bs|\lfloor i\Delta-B\rfloor-(i-1)Q|$ if and only if 
\begin{itemize}
\item $B=0,\deg\Delta=1,i\Delta\in \Z, Q\notin \Bs|i\Delta-(i-1)Q|$, or 
\item $\lfloor i\Delta-B\rfloor \sim (i-1)Q$.
\end{itemize}

\begin{lem}
Let $P_1\ne P_2,P_1\sim P_2$. Then $C\simeq \bP^1$.
\end{lem}

Below, if we write $\Delta=\delta_1 P_1+\delta_2 P_2$, we mean $P_1\ne P_2$ too.

\begin{lem} $Q\in \Bs|\lceil K+B+2L\rceil|$ if and only if one of the following holds:
\begin{itemize}
\item $\Delta=\frac{1}{2}P_1+\frac{1}{2}P_2, Q\notin \Bs|P_1+P_2-Q|,B=0$.
\item $\Delta=\frac{1}{2}P_1+\frac{1}{2}P_2,Q\sim P_1,0\ne B\le P_2$.
\item $\lfloor 2\Delta\rfloor\sim Q, B\le \min(\Delta,\{2\Delta\})$.
\end{itemize}
\end{lem}

\begin{proof}
Go through cases 1,2,3 of Theorem 4.3.
\end{proof}

\begin{lem} $Q\in \cap_{i=2}^3\Bs|\lceil K+B+iL\rceil|$ if and only if one of the following holds:
\begin{itemize}
\item $C\simeq \bP^1$, $\Delta=\frac{1}{2}P_1+\frac{1}{2}P_2,B\le \frac{1}{2}P_2$. $N=+\infty$.
\item $C\simeq \bP^1, \Delta=\frac{2}{3}P_1+\frac{1}{3}P_2, 0\ne B\le \frac{1}{3}P_1$. $N=+\infty$.
\item $\Delta=\frac{2}{3}P_1+\frac{1}{3}P_2,Q\sim P_1, B\le \frac{1}{3}P_2$.
\item $\lfloor i\Delta\rfloor\sim (i-1)Q\ (1\le i\le 3), B\le \min_{i=1}^3\{i\Delta\}$.
\end{itemize}
\end{lem}

\begin{proof}
For each solution in 4.7, go through cases 1,2,3 of Theorem 4.3.
\end{proof}

\begin{thm}\label{4.9} Let $L\sim Q-\Delta,\lfloor \Delta\rfloor=0,B\le \Delta$. Let $N\ge 3$. Then 
$Q\in \cap_{i=1}^N\Bs|\lceil K+B+iL\rceil|$ if and only if one of the following holds:
\begin{itemize}
\item[1)] $C\simeq \bP^1$, $\Delta=\delta P_1+(1-\delta)P_2,\delta\in \Z_N\cap [\frac{1}{2},\frac{N-1}{N})$,
and $B\le \frac{1}{l}P_j$ for $j=1$ or $2$, where $l=\ind(\delta)$.
\item[2)] $C\simeq \bP^1,\Delta=\frac{N-1}{N}P_1+\frac{1}{N}P_2, 0\ne B\le \frac{1}{N}P_1$.
\item[3)] $\Delta=\frac{N-1}{N}P_1+\frac{1}{N}P_2,Q\sim P_1, B\le \frac{1}{N}P_2$.
\item[4)] $\lfloor i\Delta\rfloor\sim (i-1)Q\ (1\le i\le N), B\le \min_{i=1}^N\{i\Delta\}$.
\end{itemize}
\end{thm}

\begin{proof} We use induction on $N$. Case $N=3$ do by hand.
Let $N>3$. We consider separately the solutions for $N-1$, and impose the new condition 
$Q\in \Bs|\lceil K+B+NL\rceil|$, using Theorem~\ref{sp}.

Case $(1)_{N-1},(2)_{N-1}$: here one must show $N\to \infty$ (must write down).

Case $(3)_{N-1}$: Let $\Delta=\frac{N-2}{N-1}P_1+\frac{1}{N-1}P_2,Q\sim P_1,B\le \frac{1}{N-1}P_2$.
Since $N\Delta\notin \Z$, only case 3) of Theorem~\ref{sp} may apply for $NL$. The new condition is 
$NL\sim Q-\Delta_N$, $\lfloor \Delta_N\rfloor=0,B\le \Delta_N$.
We obtain $\Delta_N\sim N\Delta-(N-1)Q$. That is $\Delta_N=\{N\Delta\}$ and $\lfloor N\Delta\rfloor\sim (N-1)Q$.
We have 
$$
N\Delta=(N-2+\frac{N-2}{N-1})P_1+(1+\frac{1}{N-1})P_2.
$$ 
Therefore $\lfloor N\Delta\rfloor=(N-2)P_1+P_2$. Then $(N-2)P_1+P_2\sim (N-1)Q$. 
Then $P_2\sim P_1$. Since $P_1\ne P_2$, we obtain $C\simeq \bP^1$. The condition 
$B\le \Delta_N$ is already satisfied. We obtain $C\simeq \bP^1,\Delta=\frac{N-2}{N-1}P_1
+\frac{1}{N-1}P_2,B\le \frac{1}{N-1}P_2$, which belongs to case $(1)_N$.

Case $(4)_{N-1}$: suppose case 3) of Theorem~\ref{sp} applies to $NL$. That is 
$NL\sim Q-\Delta_N$, $\lfloor \Delta_N\rfloor=0,B\le \Delta_N$. As above, we obtain
$\Delta_N=\{N\Delta\}$, $\lfloor N\Delta\rfloor\sim (N-1)Q$. We obtain case $(4)_N$.

Suppose now that cases 1) or 2) of Theorem~\ref{sp} apply to $NL$. That is 
$N\Delta\sim (N-1)Q+P_N$, and either $Q\ne P_N,B=0$, or $0\ne B\le P_N$. We have 
$$
N\Delta\sim (N-1)Q+P_N.
$$

Now $\Delta$ has degree one. It has at least two coefficients.

Case $\delta\ge \frac{N-1}{N}$. Then $\delta\ge \frac{N-1}{N}$. Then $\Delta=\frac{N-1}{N}P_1+\frac{1}{N}P_2$.
The conditions become $P_1\sim Q,P_N\sim P_2$.

Case $\delta<\frac{N-1}{N}$. Then $\delta+\delta'>1-\frac{1}{N+1}$. Since $N\Delta\in \Z$, we deduce
$\Delta''=0$. Therefore $\Delta=\delta P_1+(1-\delta)P_2$, and $i\delta\notin \Z$ for every $1\le i\le N-1$. 
Since $2\Delta\notin \Z$, we have $\delta\ne \frac{1}{2}$. Therefore $\delta\in (\frac{1}{2},\frac{N-1}{N})$. 
Therefore $\lfloor 2\Delta'\rfloor=0$. The condition $\lfloor 2\Delta\rfloor\sim Q$ becomes $P_1\sim Q$.
Set $j=\lceil \frac{1}{1-\delta}\rceil$, so that $\lfloor j(1-\delta)\rfloor=1$. 
We have $\delta<\frac{N-1}{N}$. Since $\delta\in \Z_N$, we obtain
$\delta\le \frac{N-2}{N-1}$. Since $(N-1)\delta\notin \Z$, we obtain $\delta<\frac{N-2}{N-1}$. Therefore 
$j\le N-1$. Then $\lfloor j\Delta\rfloor\sim (j-1)Q$ and $P_1\sim Q$ imply $P_2\sim Q$. Then 
$P_1\sim P_2$. Therefore $C\simeq \bP^1$. We obtained two cases:
\begin{itemize}
\item $\Delta=\frac{N-1}{N}P_1+\frac{1}{N}P_2,P_1\sim Q,P_N\sim P_2$.
\item $C=\bP^1$, $\Delta=\delta P_1+(1-\delta)P_2$, and $\delta<\frac{N-1}{N}$, $\delta\in \Z_N\setminus \Z_{N-1}$.
\end{itemize}
It remains to understand the condition on $B$ too. First case: $(3)_N$ or $(2)_N$. Second $(1)_N$.
\end{proof}

\begin{cor} $Q\in \Bs|\lceil K+B+iL\rceil|$ for every $i\ge 1$ if and only if one of the following holds:
\begin{itemize}
\item[1)] $L\sim Q-P\ (Q\ne P), Q\notin \cup_{i\ge 2}\Bs|iP-(i-1)Q|, B=0$.
\item[2)] $L\sim 0, 0\ne B\le P\sim Q$.
\item[3)] $C\simeq \bP^1$, $L\sim \epsilon (P_1-P_2) \ (P_1\ne P_2),\epsilon \in (0,\frac{1}{2}]$,
and either $\epsilon\notin \Q$ and $B=0$, or $\epsilon\in \Q$ and $B\le \frac{1}{l}P_j$ for $j=1$ or $2$, 
where $l\ge 1$ is minimal such that $l\epsilon\in \Z$.
\end{itemize}
\end{cor}

\begin{proof} We use the Theorem~\ref{4.9} for every $N$. The first two cases are valid for all $N$. 
Remain two cases:

3a) $C=\bP^1$, $L\sim Q-(\delta P_1+(1-\delta)P_2)$, $\delta\in [\frac{1}{2},1)\cap \Q$ and $B\le \frac{1}{l}P_j$
for some $j=1,2$, where $l$ is the index of $\delta$.

3b) $L\sim Q-\Delta$, $\lfloor \Delta\rfloor=0$, $\lfloor i\Delta\rfloor\sim (i-1)Q$ for all $i\ge 1$, 
$B\le\inf_{i\ge 1}\{i\Delta\}$. By diophantine approximation, the last condition becomes $B=0$. 
By Corollary~\ref{1.12}, $\Delta=\delta P_1+(1-\delta)P_2$ with $\delta\in [\frac{1}{2},1)\setminus\Q$. 
Finally,
$$
\lfloor i\delta\rfloor P_1+\lfloor i-i\delta\rfloor P_2\sim (i-1)Q \ (i\ge 1).
$$
Since $\delta\ne \frac{1}{2}$, we have $\delta'<\frac{1}{2}$. Therefore the condition for $i=2$ becomes 
$P_1\sim Q$. Let $j=\lceil \frac{1}{1-\delta} \rceil$. Then $\lfloor j-j\delta\rfloor=1$. The condition for $j$
becomes $P_2\sim Q$. Therefore $P_1\sim P_2$. Therefore $C\simeq \bP^1$.

Note $L\sim Q-(\delta P_1+(1-\delta)P_2)\sim (1-\delta)(P_1-P_2)$.
\end{proof}

%\clearpage
%%%%%%%%%%%%%%%%%%%%%%%%%%%%%%%%%%%%%%%
%%%%%%%%%%%%%%%%%%%%%%%%%%%%%%%%%%%%%%%

\end{document}